\documentclass{amsart}%
\usepackage{amsfonts}
\usepackage{amsmath}
\usepackage{amssymb}
\usepackage{graphicx}%
\providecommand{\U}[1]{\protect\rule{.1in}{.1in}}
\newtheorem{theorem}{Theorem}
\theoremstyle{plain}

\newtheorem{corollary}{Corollary}

\newtheorem{definition}{Definition}
\newtheorem{example}{Example}

\newtheorem{lemma}{Lemma}

\newtheorem{proposition}{Proposition}

\numberwithin{equation}{section}
\begin{document}
\title[Weakly J-ideals of Commutative Rings ]{Weakly J-ideals of Commutative Rings }
\author{Hani A. Khashan }
\address{Department of Mathematics, Faculty of Science, Al al-Bayt University,Al
Mafraq, Jordan.}
\email{hakhashan@aabu.edu.jo.}
\author{Ece Yetkin Celikel}
\address{Department of Electrical-Electronics Engineering, Faculty of Engineering,
Hasan Kalyoncu University, Gaziantep, Turkey.}
\email{ece.celikel@hku.edu.tr, yetkinece@gmail.com.}
\thanks{This paper is in final form and no version of it will be submitted for
publication elsewhere.}
\date{February, 2021}

\begin{abstract}
Let $R$ be a commutative ring with non-zero identity. In this paper, we
introduce the concept of weakly $J$-ideals as a new generalization of
$J$-ideals. We call a proper ideal $I$ of a ring $R$ a weakly $J$-ideal if
whenever $a,b\in R$ with $0\neq ab\in I$ and $a\notin J(R)$, then $a\in I$.
Many of the basic properties and characterizations of this concept are
studied. We investigate weakly $J$-ideals under various contexts of
constructions such as direct products, localizations, homomorphic images.
Moreover, a number of examples and results on weakly $J$-ideals are discussed.
Finally, the third section is devoted to the characterizations of these
constructions in an amagamated ring along an ideal.

\end{abstract}
\maketitle

\section{Introduction}

We assume throughout the whole paper, all rings are commutative with a
non-zero identity. For any ring $R$, by $U(R),$ $N(R)$ and $J(R)$, we denote
the set of all units in $R$, the nilradical and the Jacobson radical of $R,$
respectively. In 2017, Tekir et al., \cite{Tek} introduced the concept of
$n$-ideals. A proper ideal $I$ of a ring $R$ is called an $n$-ideal if
whenever $a,b\in R$ and $ab\in I$ such that $a\notin N(R)$, then $b\in I$.
Recently, Khashan and Bani-Ata in \cite{Hani} introduced the notion of
$J$-ideals as a generalization of $n$-ideals in commutative rings, as follows:
A proper ideal $I$ of a ring $R$ is called a $J$-ideal if whenever $a,b\in R$
with $ab\in I$ and $a\notin J(R)$, then $b\in I$. In \cite{Haniece}, the
authors generalized $J$-ideals and defined quasi $J$-ideals as those ideals
$I$ for which $\sqrt{I}=\left\{  x\in R:x^{n}\in I\text{ for some }n\in%
\mathbb{N}
\right\}  $ are $J$-ideals. In this paper, we define and study weakly
$J$-ideals of commutative rings as a new generalization of $J$-ideals. We call
a proper ideal $I$ of $R$ a weakly $J$-ideal if $0\neq ab\in I$ where $a,b\in
R$ and $a\notin J(R)$ imply $b\in I.$ Clearly, every $J$-ideal is a weakly
$J$-ideal. However, the converse is not true in general. Indeed, (by
definition) the zero ideal of any ring is weakly $J$-ideal but for example
$\left\langle \bar{0}\right\rangle $ is not a $J$-ideal of the ring $%
\mathbb{Z}
_{6}$. For a non-trivial example, we present Example \ref{Ex}.

Among many other results in this paper, in section 2, we start with a
characterization for quasi-local rings in terms of weakly $J$-ideals (Theorem
\ref{quasi}). Many equivalent characterizations of weakly $J$-ideals for any
commutative ring are presented in Proposition \ref{J(R)} and Theorem \ref{eq}.
As a generalization of prime ideals the concept of weakly prime ideals was
first introduced in \cite{Andw} by Anderson et al. A proper ideal of a ring
$R$ is said to be a weakly prime ideal if whenever $a,b\in R$ with $0\neq
ab\in I$, then $a\in I$ or $b\in I$. We give examples to show that weakly
prime and weakly $J$-ideals are not comparable. Then we justify the
relationships between these two concepts in Proposition \ref{wp1} and
Corollary \ref{wp2}. Further, for two weakly $J$-ideals $I_{1}$ and $I_{2}$ of
a ring $R$, we show that $I_{1}\cap I_{2}$ and $I_{1}+I_{2}$ are weakly
$J$-ideals of $R$, but $I_{1}I_{2}$ is not so (see Propositions \ref{int},
\ref{sum} and Example \ref{product}).

Recall from \cite{Bouvier} (resp. \cite{Haniece}) that a ring $R$ is called
presimplifiable (resp. quasi presimplifiable) if whenever $a,b\in R$ with
$a=ab$, then $a=0$ or $b\in U(R)$ (resp. $a\in N(R)$ or $b\in U(R)$). It is
well known from \cite{Anderson1} that presimplifiable property does not pass
in general to homomorphic images. However, we show that this holds under a
certain condition: If $I$ a weakly $J$-ideal of a (quasi) presimplifiable ring
$R$, then $R/I$ is a (quasi) presimplifiable ring (see Corollary \ref{pre} and
Proposition \ref{qpre}). Moreover, we investigate weakly $J$-ideals under
various contexts of constructions such as direct products, localizations,
homomorphic images (see Propositions \ref{cart}, \ref{S} and \ref{f}).

Let $R$ be a commutative ring with identity and $M$ an $R$-module. We recall
that $R(+)M=\left\{  (r,m):r\in R,m\in M\right\}  $ with coordinate-wise
addition and multiplication defined as $(r_{1},m_{1})(r_{2},m_{2})=(r_{1}%
r_{2},r_{1}m_{2}+r_{2}m_{1})$ is a commutative ring with identity $(1,0)$.
This ring is called the idealization of $M$. For an ideal $I$ of $R$ and a
submodule $N$ of $M$, $I(+)N$ is an ideal of $R(+)M$ if and only if
$IM\subseteq N$. Moreover, the Jacobson radical of $R(+)M$ is
$J(R(+)M)=J(R)(+)M$, \cite{Anderson2}. We clarify the relationships between
weakly $J$-ideals in a ring $R$ and in an idealization ring $R(+)M$ in Theorem
\ref{6}. The idealization can be used to extend results about ideals to
modules and to provide interesting examples of commutative rings.

In Section 3, for a ring $R$ and an ideal $J$ of $R$, we examine weakly
$J$-ideals of an amalgamated ring $R\Join^{f}J$ along $J$. Some
characterizations of (weakly) $J$-ideals of the form $I\Join^{f}J$ and
$\bar{K}^{f}$ of the amalgamation $R\Join^{f}J$ where $J\subseteq J(S)$ are
given (see Theorems \ref{9}, \ref{10} and Corollaries \ref{cj}, \ref{12}).
Finally, we give various counter examples associated with the stability of
weakly $J$-ideals in these algebraic structures (see Examples \ref{14},
\ref{15}, \ref{16}, \ref{17}).

\section{Properties of weakly $J$-ideals}

In this section, we discuss some of the basic definitions and fundamental
results concerning weakly $J$-ideal. Among many other properties, we present a
number of characterizations of such class of ideals.

\begin{definition}
Let $R$ be a ring. A proper ideal $I$ of $R$ is called a weakly $J$-ideal if
whenever $a,b\in R$ such that $0\neq ab\in I$ and $a\notin J(R)$, then $b\in
I$.
\end{definition}

Clearly any $J$-ideal is a weakly $J$-ideal. The converse is not true. For a
non trivial example we have the following:

\begin{example}
\label{Ex}Consider the idealization ring $R=%
\mathbb{Z}
(+)\left(
\mathbb{Z}
_{2}\times%
\mathbb{Z}
_{2}\right)  $ and consider the ideal $I=0(+)\left\langle (\bar{1},\bar
{0})\right\rangle $ of $R$. Then $I$ is not a $J$-ideal of $R$ since for
example, $(2,(\bar{0},\bar{0}))(0,(\bar{1},\bar{1}))=(0,(\bar{0},\bar{0}))\in
I$ and $(2,(\bar{0},\bar{0}))\notin J(R)$ but $(0,(\bar{1},\bar{1}))\notin I$.
On the other hand, $I$ is a weakly $J$-ideal of $R$. Indeed, let $(r_{1}%
,(\bar{a},\bar{b})),(r_{2},(\bar{c},\bar{d}))\in R$ such that $(0,(\bar
{0},\bar{0}))\neq(r_{1},(\bar{a},\bar{b}))(r_{2},(\bar{c},\bar{d}))\in I$ and
$(r_{1},(\bar{a},\bar{b}))\notin J(R)$. Then $r_{1}\neq0$ and $(r_{1}%
r_{2},r_{1}.(\bar{c},\bar{d})+r_{2}.(\bar{a},\bar{b}))\in I$. It follows that
$r_{1}r_{2}=0$ and $r_{1}.(\bar{c},\bar{d})+r_{2}.(\bar{a},\bar{b}%
)\in\left\langle (1,0)\right\rangle $ and so $r_{2}=0$ and $r_{1}.(\bar
{c},\bar{d})\in\left\langle (1,0)\right\rangle $. By assumption, we must also
have $r_{1}.(\bar{c},\bar{d})\neq(\bar{0},\bar{0})$. If $(\bar{c},\bar
{d})=(\bar{1},\bar{1})$ or $(\bar{0},\bar{1})$, then $r_{1}.(\bar{c},\bar
{d})\in\left\langle (1,0)\right\rangle $ if and only if $r_{1}\in\left\langle
2\right\rangle $ and so $r_{1}.(\bar{c},\bar{d})=(\bar{0},\bar{0})$, a
contradiction. Thus, $(\bar{c},\bar{d})\in\left\langle (\bar{1},\bar
{0})\right\rangle $ and $I$ is a weakly $J$-ideal of $R$.
\end{example}

However, the classes of $J$-ideals, quasi $J$-ideals and weakly $J$-ideals
coincide in any quasi local ring.

\begin{theorem}
\label{quasi}For a ring $R$, the following statements are equivalent:
\end{theorem}

\begin{enumerate}
\item $R$ is a quasi-local ring.

\item Every proper ideal of $R$ is a $J$-ideal.

\item Every proper ideal of $R$ is a quasi $J$-ideal.

\item Every proper ideal of $R$ is a weakly $J$-ideal.

\item Every proper principal ideal of $R$ is a weakly $J$-ideal.
\end{enumerate}

\begin{proof}
(1)$\Leftrightarrow$(2)$\Leftrightarrow$(3) \cite[Theorem 3]{Haniece}.

(2)$\Rightarrow$(4)$\Rightarrow$(5) Clear.

(5)$\Rightarrow$(1) Let $M$ be a maximal ideal of $R$. If $M=0$, the result
follows clearly. Otherwise, let $0\neq a\in M$. Now, $\left\langle
a\right\rangle $ is a weakly $J$-ideal and $0\neq a.1\in\left\langle
a\right\rangle $. If $a\notin J\left(  R\right)  $, then $1\in\left\langle
a\right\rangle $, a contradiction. Thus, $a\in J\left(  R\right)  $ and
$M=J\left(  R\right)  $. Therefore, $R$\ is quasi-local ring.
\end{proof}

Let $I$ be a proper ideal of $R$. We denote by $J(I)$, the intersection of all
maximal ideals of $R$ containing $I$. Next, we obtain the following
characterization for weakly $J$-ideals of $R$.

\begin{proposition}
\label{J(R)}For a proper ideal $I$ of $R$, the following statements are equivalent.
\end{proposition}

\begin{enumerate}
\item $I$ is a weakly $J$-ideal of $R$.

\item $I\subseteq J(R)$ and whenever $a,b\in R$ with $0\neq ab\in I$, then
$a\in J(I)$ or $b\in I$.
\end{enumerate}

\begin{proof}
(1)$\Rightarrow$(2) Suppose $I$ is a weakly $J$-ideal of $R$. Let $0\neq a\in
I$. Since $0\neq a\cdot1\in I$ and $1\notin I$, then $a\in J(R)$. Hence,
$I\subseteq J(R)$. The other claim of (2) follows clearly since $J(R)\subseteq
J(I)$.

(2)$\Rightarrow$(1) Suppose that $0\neq ab\in I$ and $a\notin J(R).$ Since
$I\subseteq J(R)$, we conclude that $J(I)\subseteq J(J(R))=J(R)$ and so we get
$a\notin J(I)$. Thus, $b\in I$ and $I$ is a weakly $J$-ideal of $R$.
\end{proof}

We recall that a ring $R$ is called semiprimitive if $J(R)=0$. By (2) of
Proposition \ref{J(R)}, we conclude that $0$ is the only weakly $J$-ideal in
any semiprimitive ring.

Next, we show that a weakly $J$-ideal $I$ that is not a $J$-ideal of a ring
$R$ satisfies $I^{2}=0$.

\begin{theorem}
\label{IJ}Let $I$ be a weakly $J$-ideal of a ring $R$ that is not a $J$-ideal.
Then $I^{2}=0$.
\end{theorem}

\begin{proof}
Suppose $I^{2}\neq0$. We prove that $I$ is a $J$-ideal. Let $a,b\in R$ such
that $ab\in I$ and $a\notin J(R)$. If $ab\neq0$, then $b\in I$ since $I$ is a
weakly $J$-ideal. Suppose $ab=0$. If $aI\neq0$, then $0\neq ax$ for some $x\in
I$ and so $0\neq a(b+x)\in I$. Again, since $I$ is a weakly $J$-ideal, $b+x\in
I$ and so $b\in I$. If $bI\neq0$, then $by\neq0$ for some $y\in I\subseteq
J(R)$. Since $0\neq yb=(y+a)b\in I$ and clearly $y+a\notin J(R)$, then $b\in
I$. So, we may assume that $aI=bI=0$. Since $I^{2}\neq0$, then there exist
$x,y\in I$ such that $yx\neq0$. Hence, $0\neq yx=(y+a)(x+b)\in I$ and
$y+a\notin J(R)$ imply that $x+b\in I$. Therefore, $b\in I$ and $I$ is a
$J$-ideal of $R$.
\end{proof}

However an ideal $I$ satisfies $I^{2}=0$ need not be a weakly $J$-ideal. For
example, the ideal $I=0(+)4%
\mathbb{Z}
$ of $R=%
\mathbb{Z}
(+)%
\mathbb{Z}
$ satisfies $I^{2}=0$. But, $I$ is not a weakly $J$-ideal of $R$ as
$(0,0)\neq(2,0)(0,2)\in I$ with $(2,0)\notin J(R)$ and $(0,2)\notin I$.

As a corollary of Theorem \ref{IJ}, we have:

\begin{corollary}
\label{squareI}Let $I$ be a weakly $J$-ideal of a ring $R$ that is not a
$J$-ideal. Then

\begin{enumerate}
\item $I\subseteq N(R)$.

\item Whenever $M$ is an $R$-module and $IM=M$, then $M=0$.
\end{enumerate}
\end{corollary}

In particular, if $R$ is a reduced ring, then by Corollary \ref{squareI}, a
non-zero proper ideal $I$ is a weakly $J$-ideal if and only if $I$ is a $J$-ideal.

In the following theorem, we give some other characterizations of weakly $J$-ideals.

\begin{theorem}
\label{eq}Let $I$ be a proper ideal of a ring $R.$ Then the following are equivalent:
\end{theorem}

\begin{enumerate}
\item $I$ is a weakly $J$-ideal of $R$.

\item $(I:a)=I\cup(0:a)$ for every $a\in R\backslash J(R)$.

\item $(I:a)\subseteq J(R)\cup(0:a)$ for every $a\in R\backslash I$.

\item If $a\in R$ and $K$ is an ideal of $R$ with $0\neq Ka\subseteq I$, then
$K\subseteq J(R)$ or $a\in I$.

\item If $K$ and $L$ are ideals of $R$ with $0\neq KL\subseteq I$, then
$K\subseteq J(R)$ or $L\subseteq I.$
\end{enumerate}

\begin{proof}
(1)$\Rightarrow$(2) Let $a\in R\backslash J(R)$. It is clear that
$I\cup(0:a)\subseteq(I:a)$. Let $x\in(I:a)$ so that $ax\in I$. If $ax\neq0$,
then $x\in I$ as $I$ is a weakly $J$-ideal. If $ax=0$, then $x\in(0:a)$. Thus,
$(I:a)\subseteq I\cup(0:a)$ and the equality holds.

(2)$\Rightarrow$(1) Let $a,b\in R$ such that $0\neq ab\in I$ and $a\notin
J(R)$. Then $b\notin(0:a)$ and so $b\in(I:a)\subseteq I$.

(1)$\Rightarrow$(3) Similar to the proof of (1)$\Rightarrow$(2).

(3)$\Rightarrow$(4) Suppose that $0\neq aK\subseteq I$ and $a\notin I$. Then
$(I:a)\neq(0:a)$ and so $K\subseteq(I:a)\subseteq J(R)$, as needed.

(4)$\Rightarrow$(5) Assume on the contrary that there are ideals $K$ and $L$
of $R$ such that $0\neq KL\subseteq I$ but $K\nsubseteq J(R)$ and $L\nsubseteq
I$. Since $KL\neq0$, there exists $a\in L$ such that $0\neq Ka\subseteq I$.
Since $K\nsubseteq J(R)$, we have $a\in I$ by (4). Now, choose an element
$x\in L\backslash I$. Similar to the previous argument, we conclude $Kx=0$.
(Indeed, if $Kx\neq0$, then $x\in I$). Hence, $0\neq K(a+x)\subseteq I$ and
$K\nsubseteq J(R)$ imply that $(a+x)\in I$. Thus, $x\in I$, a contradiction.

(5)$\Rightarrow$(1) Let $a,b\in R$ with $0\neq ab\in I$. Write $K=<a>$ and
$L=\left\langle b\right\rangle $. Then the result follows directly by (5).
\end{proof}

\begin{proposition}
\label{Quotient}Let $S$ be a non-empty subset of $R$. If $I$ and $(0:S)$ are
weakly $J$-ideals of $R$ where $S\nsubseteq I$, then so is $(I:S)$.
\end{proposition}

\begin{proof}
We first note that $(I:S)$ is proper in $R$ since otherwise, $S\subseteq I$, a
contradiction. Let $0\neq ab\in(I:S)$ and $a\notin J(R).$ If $abS\neq0$, then
$bS\subseteq I$ and so $b\in(I:S)$ as $I$ is weakly $J$-ideal. If $abS=0$,
then $0\neq ab\in(0:S)$ which implies $b\in(0:S)$ as $(0:S)$ is a weakly
$J$-ideal of $R$. Thus, again $b\in(I:S)$ as required.
\end{proof}

Recall that a proper ideal $P$ of a ring $R$ is called weakly prime if
whenever $a,b\in R$ such that $0\neq ab\in P$, then $a\in P$ or $b\in P$. In
general, weakly $J$-ideals and weakly prime ideals are not comparable. For
example, any non-zero prime ideal in the domain of integers is weakly prime
that is not a weakly $J$-ideal. On the other hand, the ideal $\left\langle
16\right\rangle $ is a weakly $J$-ideal in the ring $%
\mathbb{Z}
_{32}$ which is clearly not weakly prime. However, for ideals contained in the
Jacobson radical, we clarify in the following proposition that weakly prime
ideals are weakly J-ideals. The proof is straightforward.

\begin{proposition}
\label{wp1}If $I$ is a weakly prime ideal of a ring $R$ and $I\subseteq J(R)$,
then $I$ is a weakly $J$-ideal.
\end{proposition}

The converse of the previous proposition holds under certain conditions:

\begin{corollary}
\label{wp2}Let $I$ be an ideal of a ring $R$. Suppose $I$ is maximal with
respect to the property: $I$ and $(0:a)$ are weakly $J$-ideals for all
$a\notin I$. Then $I$ is weakly prime in $R$.
\end{corollary}

\begin{proof}
Let $a,b\in R$ such that $0\neq ab\in I$ and $a\notin I$. If we choose
$S=\left\{  a\right\}  $ in Proposition \ref{Quotient}, then we conclude that
$(I:a)$ is a weakly $J$-ideal of $R$. Moreover, clearly $c\notin I$ for all
$c\notin(I:a)$ and so $(0:c)$ is a weakly $J$-ideals. By maximality of $I$, we
must have $b\in(I:a)=I$ as required.
\end{proof}

\begin{proposition}
\label{int}If $\left\{  I_{i}:i\in\Delta\right\}  $ is a non empty family of
weakly $J$-ideals of a ring $R$, then $\bigcap\limits_{i\in\Delta}$ $I_{i}$ is
a weakly $J$-ideal.
\end{proposition}

\begin{proof}
Let $a,b\in R$ such that $0\neq ab\in\bigcap\limits_{i\in\Delta}$ $I_{i}$ and
$a\notin J(R)$. Since for all $i\in\Delta$, $I_{i}$ is a weakly $J$-ideal of
$R$, we have $b\in I_{i}$. Hence, $b\in\bigcap\limits_{i\in\Delta}$ $I_{i}$
and the result follows.
\end{proof}

In general, the converse of Proposition \ref{int} is not true. For example,
while $\left\langle \bar{0}\right\rangle =\left\langle \bar{2}\right\rangle
\cap\left\langle \bar{3}\right\rangle $ is a weakly $J$-ideal of $%
\mathbb{Z}
_{6}$, non of the ideals $\left\langle \bar{2}\right\rangle $ and
$\left\langle \bar{3}\right\rangle $ are (weakly) $J$-ideals.

Next, we characterize weakly $J$-ideals of a Cartesian product of two rings.

\begin{proposition}
\label{cart}Let $R=R_{1}\times R_{2}$ be a decomposable ring and $I$ be a
non-zero proper ideal of $R$. Then the following statements are equivalent.

\begin{enumerate}
\item $I$ is a weakly $J$-ideal of $R.$

\item $I=$ $I_{1}\times R_{2}$ where $I_{1}$ is a $J$-ideal of $R_{1}$ or $I=$
$R_{1}\times I_{2}$ where $I_{2}$ is a $J$-ideal of $R_{2}.$

\item $I$ is a $J$-ideal of $R.$
\end{enumerate}
\end{proposition}

\begin{proof}
(1)$\Rightarrow$(2) Let $I=I_{1}\times I_{2}$ be a non-zero weakly $J$-ideal
of $R$. Assume $I_{1}$ and $I_{2}$ are proper in $R_{1}$ and $R_{2}$,
respectively and choose $0\neq(a,b)\in I$. Then $0\neq(a,1)(1,b)\in I$ and
neither $(a,1)\in J(R)$ nor $(1,b)\in I$, a contradiction. We may assume with
no loss of generality that $I_{1}\neq R_{1}$ and $I_{2}=R_{2}.$ Since clearly
$I^{2}\neq0$, $I$ is a $J$-ideal of $R$ by Corollary \ref{squareI} (2). Let
$a,b\in R_{1}$ such that $ab\in I_{1}$ and $a\notin J(R_{1})$. Then
$(a,0)(b,0)\in I$ and $(a,0)\notin J(R)$ imply that $(b,0)\in I$ and so $b\in
I_{1}$ as required.

(2)$\Rightarrow$(3) We may assume that $I=$ $I_{1}\times R_{2}$ where $I_{1}$
is a $J$-ideal of $R_{1}$. Suppose that $(a,b)(c,d)\in I$ and $(a,b)\notin
J(R).$ Then clearly $a\notin J(R_{1})$ and $ac\in I_{1}$ which imply $c\in
I_{1}$. Thus, $(c,d)\in I$ and we are done.

(3)$\Rightarrow$(1) straightforward.
\end{proof}

\begin{corollary}
Let $R=R_{1}\times R_{2}$ be a decomposable ring. If $I$ is a weakly $J$-ideal
of $R$ that is not a $J$-ideal, then $I=0$.
\end{corollary}

Let $I$ be a proper ideal of $R.$ In the following, the notation $Z_{I}%
(R)$\ denotes the set of $\{r\in R|rs\in I$ for some $s\in R\backslash I\}.$

\begin{proposition}
\label{S}Let $S$ be a multiplicatively closed subset of a ring $R$ such that
$J(S^{-1}R)=S^{-1}J(R)$.
\end{proposition}

\begin{enumerate}
\item If $I$ is a weakly $J$-ideal of $R$ such that $I\cap S=\emptyset$, then
$S^{-1}I$ is a weakly $J$-ideal of $S^{-1}R$.

\item If $S^{-1}I$ is a weakly $J$-ideal of $S^{-1}R$ and $S\cap Z(R)=S\cap
Z_{I}(R)=S\cap Z_{J(R)}(R)=\emptyset$, then $I$ is a weakly $J$-ideal of $R$.
\end{enumerate}

\begin{proof}
(1)\ Let $0\neq\frac{a}{s_{1}}\frac{b}{s_{2}}\in S^{-1}I$ and $\frac{a}{s_{1}%
}\notin J(S^{-1}R)$ for some $\frac{a}{s_{1}},\frac{b}{s_{2}}\in S^{-1}R.$
Then $0\neq uab\in I$ for some $u\in S$. Since clearly $a\notin J(R)$ and $I$
is a weakly $J$-ideal, we have $ub\in I$. Hence $\frac{b}{s_{2}}=\frac
{ub}{us_{2}}\in S^{-1}I$, as needed.

(2) Let $a,b\in R$ and $0\neq ab\in I$. Then $\frac{a}{1}\frac{b}{1}\in
S^{-1}I$. If $\frac{a}{1}\frac{b}{1}=0$, then $uab=0$ for some $u\in S.$ Since
$S\cap Z(R)=\emptyset$, we have $ab=0$, a contradiction. Since $S^{-1}I$ is a
weakly $J$-ideal of $S^{-1}R$, $0\neq\frac{a}{1}\frac{b}{1}\in S^{-1}I$
implies either $\frac{a}{1}\in J(S^{-1}R)=S^{-1}J(R)$ or $\frac{b}{1}\in
S^{-1}I$. If $\frac{a}{1}\in S^{-1}J(R)$, then there exists $u\in S$ with
$ua\in J(R)$. Since $S\cap Z_{J(R)}(R)=\emptyset$, then $a\in J(R)$. If
$\frac{b}{1}\in S^{-1}I$, then there exists $v\in S$ with $vb\in I$ and so
$b\in I$ as $S\cap Z_{I}(R)=\emptyset$. Therefore, $I$ is a weakly $J$-ideal
of $R$.
\end{proof}

\begin{proposition}
\label{f}Let $f:R_{1}\rightarrow R_{2}$ be a ring homomorphism. Then the
following statements hold.
\end{proposition}

\begin{enumerate}
\item If $f$ is a monomorphism and $I_{2}$ is a weakly $J$-ideal of $R_{2}$,
then $f^{-1}(I_{2})$ is a weakly $J$-ideal of $R_{1}.$

\item If $f$ is an epimorphism and $I_{1}$ is a weakly $J$-ideal of $R_{1}$
with $Ker(f)\subseteq I_{1}$, then $f(I_{1})$ is a weakly $J$-ideal of
$R_{2}.$
\end{enumerate}

\begin{proof}
(1) Suppose that $a,b\in R_{1}$ with $0\neq ab\in f^{-1}(I_{2})$ and $a\notin
J(R_{1})$. Observe that $f(a)\notin J(R_{2})$ by the proof of
\cite[Proposition 2.23 (2)]{Hani}. Since $Ker(f)=0$, we have $0\neq
f(ab)=f(a)f(b)\in I_{2}.$ Since $I_{2}$ is a weakly $J$-ideal of $R_{2}$, we
get $f(b)\in I_{2},$ and so $b\in f^{-1}(I_{2})$.

(2) Let $a,b\in R_{2}$ and $0\neq ab\in f(I_{1})$. Since $f$ is onto, $a=f(x)$
and $b=f(y)$ for some $x,y\in R_{1}$. Hence, $0\neq f(x)f(y)=f(xy)\in
f(I_{1})$. Since $Ker(f)\subseteq I_{1}$, we have $0\neq xy\in I_{1}$ which
implies $x\in J(R_{1})$ or $y\in I_{1}$. Thus, $a=f(x)\in J(R_{2})$ by
\cite[Lemma 2.22]{Hani} or $b=f(y)\in f(I_{1})$ and we are done.
\end{proof}

\begin{corollary}
\label{c/}Let $I$ and $K$ be proper ideals of $R$ with $K\subseteq I$.
\end{corollary}

\begin{enumerate}
\item If $I$ is a weakly $J$-ideal of $R$, then $I/K$ is a weakly $J$-ideal of
$R/K$.

\item If $K$ is a $J$-ideal of $R$ and $I/K$ is a weakly $J$-ideal of $R/K$,
then $I$ is a $J$-ideal of $R$.

\item If $K$ is a weakly $J$-ideal of $R$ and $I/K$ is a weakly $J$-ideal of
$R/I$, then $I$ is a weakly $J$-ideal of $R$.
\end{enumerate}

\begin{proof}
(1) Follows by Proposition \ref{f}.

(2) Let $a,b\in R$ with $ab\in I$ and $a\notin J(R)$. If $ab\in K$, then $b\in
K\subseteq I$. Now, suppose that $ab\notin K$. Since $K$ is a $J$-ideal,
$K\subseteq J(R)$ and so clearly $a+K\notin J(R/K)$. Since $K\neq
(a+K)(b+K)=ab+K\in I/K$ and $I/K$ is weakly $J$-ideal, we have $(b+K)\in I/K$.
Thus, $b\in I$ and we are done.

(3) Similar to (2).
\end{proof}

Recall from \cite{Bouvier} that a ring $R$ is called presimplifiable if
whenever $a,b\in R$ with $a=ab$, then $a=0$ or $b\in U(R)$. Equivalently, $R$
is presimplifiable if and only if $Z(R)\subseteq J(R)$. The next result states
that in a presimplifiable ring, weakly $J$ -ideals and $J$-ideals coincide.

\begin{proposition}
\label{prese}Every weakly $J$-ideal of a presimplifiable ring is a $J$-ideal.
\end{proposition}

\begin{proof}
Let $R$ be a presimplifiable ring and $I$ be a weakly $J$-ideal of $R.$ Then
$0$ is a $J$-ideal by \cite[Corollary 5]{Haniece}. So, the claim follows from
Corollary \ref{c/} (2).
\end{proof}

It is well known that presimplifiable property does not pass in general to
homomorphic images, \cite{Anderson1}. In view of Proposition \ref{prese} and
\cite[Theorem 8]{Haniece}, we prove that this holds under a certain condition.

\begin{corollary}
\label{pre}If $R$ is a presimplifiable ring and $I$ is a weakly $J$-ideal of
$R$, then $R/I$ is presimplifiable.
\end{corollary}

Recall from \cite{Haniece} that a proper ideal $I$ of a ring $R$ is said to be
quasi $J$-ideal if $\sqrt{I}$ is a $J$-ideal of $R$. A ring $R$ is called
quasi presimplifiable if whenever $a,b\in R$ with $a=ab$, then $a\in N(R)$ or
$b\in U(R)$. We need the following lemma which justifies the relation between
these two concepts.

\begin{lemma}
\label{qp}\cite{Haniece} Let $I$ be a proper ideal of a ring $R$. Then $I$ is
a quasi $J$-ideal of $R$ if and only if $I\subseteq J(R)$ and $R/I$ is quasi presimplifiable.
\end{lemma}

\begin{proposition}
\label{qpre}Let $R$ be a quasi presimplifiable ring and $I$ a weakly $J$-ideal
of $R$. Then $R/I$ is a quasi presimplifiable ring.
\end{proposition}

\begin{proof}
Suppose $ab\in\sqrt{I}$ and $a\notin J(R)$. Then $a^{n}b^{n}\in I$ for some
$n\in%
\mathbb{N}
$. Suppose $a^{n}b^{n}=0$. Since $R$ is quasi presimplifiable, then $0$ is a
quasi $J$-ideal of $R$ by Lemma \ref{qp} and so $N(R)$ is a $J$-ideal. Hence,
$ab\in N(R)$ implies $b\in N(R)\subseteq\sqrt{I}$. Now, suppose that $0\neq
a^{n}b^{n}\in I$. Since clearly $a^{n}\notin J(R),$ we conclude that $b^{n}\in
I$ and $b\in\sqrt{I}$. Hence $\sqrt{I}$ is a $J$-ideal and so $I$ is a quasi
$J$-ideal of $R$. Therefore, $R/I$ is a quasi presimplifiable ring by Lemma
\ref{qp}.
\end{proof}

\begin{proposition}
Let $R$ be a Noetherian domain and $I$ be a proper ideal of $R$. Then $I$ is a
$J$-ideal of $R$ if and only if $I\subseteq J(R)$ and $I/I^{n}$ is a weakly
$J$-ideal of $R/I^{n}$ for all positive integers $n.$
\end{proposition}

\begin{proof}
Suppose $I$ is a $J$-ideal of $R$. Then $I\subseteq J(R)$ by Proposition
\ref{J(R)} and $I/I^{n}$ is a weakly $J$-ideal of $R/I^{n}$ by Corollary
\ref{c/} (1). Conversely, suppose that for all $n\in%
\mathbb{N}
$, $I/I^{n}$ is a weakly $J$-ideal of $R/I^{n}$ and let $ab\in I$. If
$ab\notin I^{n}$ for some $n\geq2$, then clearly $I^{n}\neq(a+I^{n}%
)(b+I^{n})\in I/I^{n}$ which implies $(a+I^{n})\in J(R/I^{n})$ or
$(b+I^{n})\in I/I^{n}.$ Since by assumption, $I^{n}\subseteq I\subseteq J(R)$,
then $J(R/I^{n})=J(R)/I^{n}$. Thus, $a\in J(R)$ or $b\in I$, as needed. Now,
assume that $ab\in I^{n}$ for all $n.$ Since $R$ is Noetherian, the Krull's
intersection theorem implies that $\bigcap\limits_{n=1}^{\infty}I^{n}=0$.
Therefore, $ab=0$ and since $R$ is a domain, we conclude $a=0$ or $b=0$ and we
are done.
\end{proof}

\begin{definition}
Let $I$ be a non-zero ideal of a ring $R$. An element $a+I\in R/I$ is called a
strong zero divisor in $R/I$ if there exists $I\neq b+I\in R/I$ such that
$(a+I)(b+I)=I$ and $ab\neq0$.
\end{definition}

It is clear that any strong zero divisor in $R/I$ is a zero divisor. The
converse is not true since for example $\bar{2}+\left\langle \bar
{4}\right\rangle $ is a zero divisor in $%
\mathbb{Z}
_{8}/\left\langle \bar{4}\right\rangle $ which is not a strong zero divisor.

For an ideal $I$ of a ring $R$, we denote the set of strong zero divisors of
$R/I$ by $SZ(R/I)$. It is clear that if $I=0$, (e.g. $R$ is a field), then
$SZ(R/I)=\phi$.

Let $I$ be a non-zero ideal of a ring $R$. Analogous to to the presimplifiable
rings, we define a ring $R/I$ to be $S$-presimplifiable if $SZ(R/I)\subseteq
J(R/I)$. Next, we characterize non-zero weakly $J$-ideals in terms of
$S$-presimplifiable quotient rings.

\begin{theorem}
Let $I$ be a non-zero ideal of a ring $R$. Then $I$ is a weakly $J$-ideal of
$R$ if and only if $I\subseteq J(R)$ and $R/I$ is $S$-presimplifiable.
\end{theorem}

\begin{proof}
Suppose $I$ is a weakly $J$-ideal of $R$ and note that $I\subseteq J(R)$ by
Proposition \ref{J(R)}. Let $a+I\in SZ(R/I)$ and choose $I\neq b+I\in R/I$
such that $(a+I)(b+I)=I$ and $ab\neq0$. Then $0\neq ab\in I$ and $b\notin I$.
Hence, $a\in J(R)$ as $I$ is a weakly $J$-ideal. Therefore, $(a+I)\in
J(R)/I=J(R/I)$ and we are done. Conversely, let $a,b\in R$ such that $0\neq
ab\in I$ and $b\notin I$. Then clearly, $a+I$ is a strong zero divisor in
$R/I$ and so $a+I\in J(R/I)$. It follows that $a\in J(R)$ and so $I$ is a
weakly $J$-ideal of $R$.
\end{proof}

It is clear that for a non zero ideal $I$ of a ring $R$, if $R/I$ is
presimplifiable, then it is $S$-presimplifiable. However, we have seen in
Example \ref{Ex} that $0(+)\left\langle (\bar{1},\bar{0})\right\rangle $ is a
weakly $J$-ideal of $%
\mathbb{Z}
(+)\left(
\mathbb{Z}
_{2}\times%
\mathbb{Z}
_{2}\right)  $ that is not a $J$-ideal. In view of the above theorem and
\cite[Theorem 8]{Haniece}, we conclude that $%
\mathbb{Z}
(+)\left(
\mathbb{Z}
_{2}\times%
\mathbb{Z}
_{2}\right)  /0(+)\left\langle (\bar{1},\bar{0})\right\rangle $ is an
$S$-presimplifiable ring that is not presimplifiable.

It is well known that for any ring $R$, $J(R[\left\vert x\right\vert
])=J(R)+xR[\left\vert x\right\vert ]$.

\begin{proposition}
Let $R$ be a ring. If $I$ is a weakly $J$-ideal of $R\left[  \left\vert
x\right\vert \right]  $ (resp., $R[x]$), then $I\cap R$ is a weakly $J$-ideal
of $R$.
\end{proposition}

\begin{proof}
(1) Suppose $I$ is a weakly $J$-ideal of $R\left[  \left\vert x\right\vert
\right]  $. Let $0\neq ab\in I\cap R$ and $a\notin J(R)$ for $a,b\in R$. Then
$0\neq ab\in I$ and $a\notin J(R\left[  \left\vert x\right\vert \right]  )$
imply that $b\in I$. Thus, $b\in I\cap R$ as needed.
\end{proof}

A proper ideal $I$ in a ring $R$ is called superfluous if whenever $J$ is an
ideal of $R$ with $I+J=R$, then $J=R$.

\begin{lemma}
\label{sup}If an ideal $I$ of a ring $R$ is a weakly $J$-ideal, then it is superfluous.
\end{lemma}

\begin{proof}
Suppose $I+J=R$ for some ideal $J$ of $R$. Then $1=x+y$ for some $x\in I$ and
$y\in J$ and so $1-y\in I\subseteq J(R)$ by Proposition \ref{J(R)}. Thus $y\in
J$ is a unit and $J=R$.
\end{proof}

\begin{proposition}
\label{sum}Let $I_{1}$ and $I_{2}$ be weakly $J$-ideals of a ring $R$. Then
$I_{1}+I_{2}$ is a weakly $J$-ideal of $R$.
\end{proposition}

\begin{proof}
Suppose that $I_{1}$ and $I_{2}$ are weakly $J$-ideals. Then $I_{1}+I_{2}$ is
proper by Lemma \ref{sup}. Since $I_{1}\cap I_{2}$ is a weakly $J$-ideal by
Proposition \ref{int}, then $I_{1}/(I_{1}\cap I_{2})$ is a weakly $J$-ideal of
$R/(I_{1}\cap I_{2})$ by Corollary \ref{c/} (1). From the isomorphism
$I_{1}/(I_{1}\cap I_{2})\cong(I_{1}+I_{2})/I_{2}$, we conclude that
$(I_{1}+I_{2})/I_{2}$ is a weakly $J$-ideal of $R/I_{2}$. Thus, $I_{1}+I_{2}$
is a weakly $J$-ideal of $R$ by Corollary \ref{c/} (3).
\end{proof}

Next, we generalize the concept of $J$-multiplicatively closed subset of a
ring $R$, \cite[Definition 2.27]{Hani}.

\begin{definition}
Let $S$ be a non empty subset of a ring $R$ such that $R-J(R)\subseteq S$.
Then $S$ is called weakly $J$-multiplicatively closed if $ab\in S$ or $ab=0$
for all $a\in R-J(R)$ and $b\in S$.
\end{definition}

Similar to the relation between $J$-ideals and $J$-multiplicatively closed
subsets of rings, we have:

\begin{proposition}
An ideal $I$ is a weakly $J$-ideal of a ring $R$ if and only if $R-I$ is a
weakly $J$-multiplicatively closed subset of $R$.
\end{proposition}

\begin{proof}
If $I$ is a weakly $J$-ideal of $R$, then $I\subseteq J(R)$ and so
$R-J(R)\subseteq R-I$. Let $a\in R-J(R)$ and $b\in R-I$. If $ab=0$, then we
are done. Otherwise, suppose $ab\neq0$. Since $I$ is a weakly $J$-ideal, then
$ab\in R-I$ and so $R-I$ is a weakly $J$-multiplicatively closed subset of
$R$. Conversely, suppose $R-I$ is a $J$-multiplicatively closed subset of $R$.
Let $a,b\in R$ such that $0\neq ab\in I$ and $a\notin J(R)$. If $b\in R-I$,
Then $ab\in R-I$ as $R-I$ is a weakly $J$-multiplicatively closed subset. This
contradiction implies $b\in I$ and so $I$ is a weakly $J$-ideal of $R$.
\end{proof}

\begin{proposition}
Let $S$ be a weakly $J$-multiplicatively closed subset of a ring $R$ such that
$S\cap\bigcup\limits_{a\notin J(R)}(0:a)=\phi$. If an ideal $I$ of $R$ is
maximal with respect to the property $I\cap S=\phi$, then $I$ is a weakly
$J$-ideal of a ring $R$.
\end{proposition}

\begin{proof}
Suppose $I$ is not a weakly $J$-ideal of $R$. Then there are $a\notin J(R)$
and $b\notin I$ such that $ab\in I$ but $ab\neq0$. Since $I\varsubsetneq
(I:a)$, then $(I:a)\cap S\neq\phi$ and so there exists $s\in(I:a)\cap S$. Now,
$as\in I$ and since $S$ is weakly $J$-multiplicatively closed, we have either
$as\in S$ or $as=0$. If $as\in S$, then $S\cap I\neq\phi$, a contradiction. If
$as=0$, then $s\in S\cap\bigcup\limits_{a\notin J(R)}(0:a)$ which is also a
contradiction. Therefore, $I$ is a weakly $J$-ideal of a ring $R$.
\end{proof}

Next, we justify the relation between weakly $J$-ideals of a ring $R$ and
those of the idealization ring $R(+)M$.

\begin{theorem}
\label{6}Let $I$ be an ideal of a ring $R$ and $N$ be a submodule of an
$R$-module $M$.

\begin{enumerate}
\item If $I(+)N$ is a weakly $J$-ideal of the idealization ring $R(+)M$, then
$I$ is a weakly $J$-ideal of $R$.

\item $I(+)M$ is a weakly $J$-ideal of $R(+)M$ if and only if $I$ is a weakly
$J$-ideal of $R$ and for $x,y\in R$ with $xy=0$ but $x\notin J(R)$ and
$y\notin I$, $x,y\in Ann(M)$.
\end{enumerate}
\end{theorem}

\begin{proof}
(1) If $I=R$, then clearly $I(+)N=R(+)M$, a contradiction. Let $a,b\in R$ with
$0\neq ab\in I$ and $a\notin J(R)$. Then $(0,0)\neq(a,0)(b,0)\in I(+)N$ and
$(a,0)\notin J(R(+)M)=J(R)(+)M$. Since $I(+)N$ is a weakly $J$-ideal, we have
$(b,0)\in I(+)N$ and $b\in I$ as needed.

(2) Suppose $I(+)M$ is a weakly $J$-ideal. Then $I$ is so by (1). Now, for
$x,y\in R$, suppose $xy=0$ but $x\notin J(R)$ and $y\notin I$. If $x\notin
Ann(M)$, then there exists $m\in M$ such that $xm\neq0$. Hence, $(0,0)\neq
(x,0)(y,m)\in I(+)M$ but $(x,0)\notin J(R(+)M)$ and $(y,m)\notin I(+)M$, a
contradiction. Therefore, $x\in Ann(M)$. Similarly, we can prove that $y\in
Ann(M)$. Conversely, suppose $(0,0)\neq\left(  a,m_{1}\right)  \left(
b,m_{2}\right)  \in$ $I(+)M$ and $\left(  a,m_{1}\right)  \notin J(R(+)M)$ for
$\left(  a,m_{1}\right)  ,\left(  b,m_{2}\right)  \in$ $R(+)M$. Then $ab\in I$
and $a\notin J(R)$. If $ab\neq0$, then $b\in I$ as $I$ is a weakly $J$-ideal
and so $\left(  b,m_{2}\right)  \in I(+)M$. Suppose $ab=0$ but neither $a\in
J(R)$ nor $b\in I$. By assumption, $a,b\in Ann(M)$ and so $\left(
a,m_{1}\right)  \left(  b,m_{2}\right)  =(0,0)$, a contradiction.
\end{proof}

However, in general, if $I$ is a weakly $J$-ideal of $R$, then $I(+)M$ need
not be so. For example, although $0$ is a (weakly) $J$-ideal of $%
\mathbb{Z}
$, the ideal $0(+)\left\langle 4\right\rangle $ of $%
\mathbb{Z}
(+)%
\mathbb{Z}
$ is not a weakly $J$-ideal. Indeed, $(0,0)\neq(2,2)(0,2)\in0(+)\left\langle
4\right\rangle $ but $(2,2)\notin J(%
\mathbb{Z}
(+)%
\mathbb{Z}
)$ and $(0,2)\notin0(+)\left\langle 4\right\rangle $.

\section{(Weakly) $J$-ideals of Amalgamated Rings Along an Ideal}

Let $R$ and $S$ be two rings, $J$ be an ideal of $S$ and $f:R\rightarrow S$ be
a ring homomorphism. The set $R\Join^{f}J=\left\{  (r,f(r)+j):r\in R\text{,
}j\in J\right\}  $ is a subring of $R\times S$ (with identity element
$(1_{R},1_{S})$ ) called the amalgamation of $R$ and $S$ along $J$ with
respect to $f$. In particular, if $Id_{R}:R\rightarrow R$ is the identity
homomorphism on $R$, then $R\Join J=R\Join^{Id_{R}}J=\left\{  (r,r+j):r\in
R\text{, }j\in J\right\}  $ is the amalgamated duplication of a ring along an
ideal $J$. This construction has been first defined and studied by D'Anna and
Fontana, \cite{DAnna1}. Many properties of this ring have been investigated
and analyzed over the last two decades, see for example \cite{DAnna2},
\cite{DAnna3}.

Let $I$ be an ideal of $R$ and $K$ be an ideal of $f(R)+J$. Then $I\Join
^{f}J=\left\{  (i,f(i)+j):i\in I\text{, }j\in J\right\}  $ and $\bar{K}%
^{f}=\{(a,f(a)+j):a\in R$, $j\in J$, $f(a)+j\in K\}$ are ideals of $R\Join
^{f}J$, \cite{DAnna3}.

\begin{lemma}
\label{8}\cite{DAnna3}Let $R$, $S$, $J$ and $f$ be as above. The set of all
maximal ideals of $R\Join^{f}J$ is $Max(R\Join^{f}J)=\left\{  M\Join^{f}J:M\in
Max(R)\right\}  \cup\left\{  \bar{Q}^{f}:Q\in Max(S)\backslash V(J)\right\}  $
where $V(J)$ denotes the set of all prime ideals containing $J$.
\end{lemma}

In particular if $J\subseteq J(S)$ (e.g. $S$ is quasi-local or $J$ is a weakly
$J$-ideal), then we conclude from Lemma \ref{8} that $J(R\Join^{f}%
J)=J(R)\Join^{f}J$. The proof of the following proposition is straightforward
by using Theorem \ref{quasi}.

\begin{proposition}
\label{pp}Consider the amalgamation of\ rings $R$ and $S$ along the ideal $J$
of $S$ with respect to a homomorphism $f$. If $R$ is a quasi-local ring and
$J\subseteq J(S)$, then every ideal of $R\Join^{f}J$ is a (weakly) $J$-ideal.
\end{proposition}

Next, we give a characterization of (weakly) $J$-ideals of the form
$I\Join^{f}J$ and $\bar{K}^{f}$ of the amalgamation $R\Join^{f}J$ when
$J\subseteq J(S)$.

\begin{theorem}
\label{9}Consider the amalgamation of rings $R$ and $S$ along the ideals $J$
of $S$ with respect to a homomorphism $f$. Let $I$ be an ideal of $R$. Then

\begin{enumerate}
\item If $I\Join^{f}J$ is a $J$-ideal of $R\Join^{f}J$, then $I$ is a
$J$-ideal of $R$. Moreover, the converse is true if $J\subseteq J(S)$.

\item If $I\Join^{f}J$ is a weakly $J$-ideal of $R\Join^{f}J$, then $I$ is a
weakly $J$-ideal of $R$ and for $a,b\in R$ with $ab=0$, but $a\notin J(R)$,
$b\notin I$, then $f(a)j+f(b)i+ij=0$ for every $i,j\in J$. Moreover, the
converse is true if $J\subseteq J(S)$.
\end{enumerate}
\end{theorem}

\begin{proof}
(1) Suppose $I\Join^{f}J$ is a $J$-ideal of $R\Join^{f}J$. Let $a,b\in R$ such
that $ab\in I$ and $a\notin J(R)$. Then $(a,f(a))(b,f(b))\in I\Join^{f}J$ and
$(a,f(a))\notin J(R\Join^{f}J)$ since otherwise $a\in M$ for each $M\in
Max(R)$ by Lemma \ref{8}, a contradiction. It follows that $(b,f(b))\in
I\Join^{f}J$ and so $b\in I$ as needed.

Moreover, suppose $J\subseteq J(S)$ and $I$ is a a $J$-ideal of $R$. Let
$(a,f(a)+j_{1})(b,f(b)+j_{2})=(ab,$ $f(ab)+f(a)j_{2}+f(b)j_{1}+j_{1}j_{2})\in
I\Join^{f}J$ for $(a,f(a)+j_{1}),(b,f(b)+j_{1})\in R\Join^{f}J$. If
$(a,f(a)+j_{1})\notin J(R\Join^{f}J)=J(R)\Join^{f}J$, then $a\notin J(R)$.
Since $ab\in I$, we conclude that $b\in I$ and so $(b,f(b)+j_{2})\in
I\Join^{f}J$. Thus, $I\Join^{f}J$ is a $J$-ideal of $R\Join^{f}J$.

(2) Suppose $I\Join^{f}J$ is a weakly $J$-ideal of $R\Join^{f}J$ and let
$a,b\in R$ such that $0\neq ab\in I$ and $a\notin J(R)$. Then $(0,0)\neq
(a,f(a))(b,f(b))\in I\Join^{f}J$ and $(a,f(a))\notin J(R\Join^{f}J)$ by Lemma
\ref{8}. It follows that $(b,f(b))\in I\Join^{f}J$ and so $b\in I$. For the
second claim, suppose there exist $i,j\in J$ such whenever $a,b\in R$ with
$ab=0$, but $a\notin J(R)$, $b\notin I$, $f(a)j+f(b)i+ij\neq0$. Then
$(0,0)\neq(a,f(a)+i)(b,f(b)+j)=$
$(ab,f(ab)+f(a)j+f(b)i+ij)=(0,f(a)j+f(b)i+ij)\in I\Join^{f}J$. This is a
contradiction since $I\Join^{f}J$ is a weakly $J$-ideal, $(a,f(a)+i)\notin
J(R\Join^{f}J)$ and $(b,f(b)+j)\notin I\Join^{f}J$. Now, we prove the converse
under the assumption $J\subseteq J(S)$. Let $0\neq(a,f(a)+j_{1})(b,f(b)+j_{2}%
)=(ab,f(ab)+f(a)j_{2}+f(b)j_{1}+j_{1}j_{2})\in I\Join^{f}J$ where
$(a,f(a)+j_{1})\notin J(R\Join^{f}J)=J(R)\Join^{f}J$. Then $ab\in I$ and we
have two cases:

\textbf{Case I:} If $ab\neq0$, then as clearly $a\notin J(R)$, we have $b\in
I$. Therefore, $(b,f(b)+j_{2})\in I\Join^{f}J$ and $I\Join^{f}J$ is a weakly
$J$-ideal of $R\Join^{f}J$.

\textbf{Case II:} Suppose $ab=0$. If $a\notin J(R)$ and $b\notin I$, then by
assumption, $f(a)j+f(b)i+ij=0$ for every $i,j\in J$. This implies
$(a,f(a)+j_{1})(b,f(b)+j_{2})=(0,0)$, a contradiction. Thus, $a\in J(R)$ or
$b\in I$ and so $(a,f(a)+j_{1})\in J(R\Join^{f}J)$ or $(b,f(b)+j_{2})\in
I\Join^{f}J$ as required.
\end{proof}

\begin{corollary}
\label{cj}Consider the amalgamation of rings $R$ and $S$ along the ideal
$J\subseteq J(S)$ of $S$ with respect to a homomorphism $f$. The $J$-ideals of
$R\Join^{f}J$ containing $\{0\}\times J$ are of the form $I\Join^{f}J$ where
$I$ is a $J$-ideal of $R.$
\end{corollary}

\begin{proof}
First, we note that $I\Join^{f}J$ is a $J$-ideal of $R\Join^{f}J$ for any
$J$-ideal $I$ of $R$ by Theorem \ref{9}. Let $K$ be a $J$-ideal of $R\Join
^{f}J$ containing $\{0\}\times J.$ Consider the surjective homomorphism
$\varphi:R\Join^{f}J\rightarrow R$ defined by $\varphi(a,f(a)+j)=a$ for all
$(a,f(a)+j)\in R\Join^{f}J$. Since $Ker(\varphi)=\{0\}\times J\subseteq K$,
then $I:=\varphi(K)$ is a $J$-ideal of $R$ by \cite[Proposition 2.23
(1)]{Hani} . Since $\{0\}\times J\subseteq K$, we conclude that $K$ is of the
form $I\Join^{f}J$.
\end{proof}

As a particular case of Theorem \ref{9}, we have the following immediate corollary.

\begin{corollary}
\label{11}Let $R$ be a ring and $I$, $J$ be be proper ideals of $R$. Then

\begin{enumerate}
\item If $I\Join J$ is a $J$-ideal of $R\Join J$, then $I$ is a $J$-ideal of
$R$. Moreover, the converse is true if $J\subseteq J(R)$.

\item If $I\Join J$ is a weakly $J$-ideal of $R\Join J$, then $I$ is a weakly
$J$-ideal of $R$ and for $a,b\in R$ with $ab=0$, but $a\notin J(R)$, $b\notin
I$, then $aj+bi+ij=0$ for every $i,j\in J$. Moreover, the converse is true if
$J\subseteq J(R)$.
\end{enumerate}
\end{corollary}

\begin{theorem}
\label{10}Consider the amalgamation of \ rings $R$ and $S$ along the maximal
ideal $J$ of $S$ with respect to an epimorphism $f$. Let $K$ be an ideal of
$S$.

\begin{enumerate}
\item If $\bar{K}^{f}$ is a $J$-ideal of $R\Join^{f}J$, then $K$ is a a
$J$-ideal of $S$. The converse is true if $f(J(R))=J(S)+J$ and
$Ker(f)\subseteq J(R)$.

\item If $\bar{K}^{f}$ is a weakly $J$-ideal of $R\Join^{f}J$, then $K$ is a
weakly $J$-ideal of $S$ and when $f(a)+j\notin J(S),f(b)+k\notin K$ with
$a,b\in R$, $j,k\in J$ and $(f(a)+j)(f(b)+k)=0$, then $ab=0$. The converse is
true if $f(J(R))=J(S)+J$ and $Ker(f)\subseteq J(R)$.
\end{enumerate}
\end{theorem}

\begin{proof}
(1) Suppose $\bar{K}^{f}$ is a $J$-ideal of $R\Join^{f}J$. Let $x,y\in S$,
say, $x=f(a)$ and $y=f(b)$ for $a,b\in R$. Suppose $xy\in K$ and $x\notin
J(S)$. Then $(a,f(a)),(b,f(b))\in R\Join^{f}J$ such that
$(a,f(a))(b,f(b))=(ab,(f(a))(f(b)))\in\bar{K}^{f}$. If $(a,f(a))\in
J(R\Join^{f}J)$, then $(a,f(a))\in\bar{Q}^{f}$ for all $Q\in Max(S)\backslash
V(J)$. Moreover, since $J$ is maximal in $S$, then $f^{-1}(J)$ is maximal in
$R$ and so $(a,f(a))\in f^{-1}(J)\Join^{f}J$. Thus, $f(a)\in J=V(J)$ and so
$f(a)\in Q$ for all $Q\in Max(S)$, a contradiction. Therefore, $(a,f(a))\notin
J(R\Join^{f}J)$ and so $(b,f(b))\in\bar{K}^{f}$ as $\bar{K}^{f}$ is a
$J$-ideal of $R\Join^{f}J$. Hence, $y=f(b)\in K$ and we are done.

Now, suppose $f(J(R))=J(S)+J$, $Ker(f)\subseteq J(R)$ and $K$ is a $J$-ideal
of $S$. Let $(a,f(a)+j),(b,f(b)+k)\in R\Join^{f}J$ such that
$(a,f(a)+j)(b,f(b)+k)=(ab,(f(a)+j)(f(b)+k))\in\bar{K}^{f}$ and
$(a,f(a)+j)\notin J(R\Join^{f}J)$. We claim that $f(a)+j\notin J(S)$. Suppose
not. Then $f(a)\in J(S)+J=f(J(R))$ and so $a\in J(R)$ as $Ker(f)\subseteq
J(R)$. Thus, $(a,f(a)+j)\in\left\{  M\Join^{f}J:M\in Max(R)\right\}  $.
Moreover, $f(a)+j\in Q$ for all $Q\in Max(S)$ implies that $(a,f(a)+j)\in
\left\{  \bar{Q}^{f}:Q\in Max(S)\backslash V(J)\right\}  $. It follows by
Lemma \ref{8} that $(a,f(a)+j)\in J(R\Join^{f}J)$, a contradiction. Since
$(f(a)+j)(f(b)+k)\in K$ and $K$ is a $J$-ideal of $S$, then $f(b)+k\in K$.
Hence, $(b,f(b)+k)\in\bar{K}^{f}$ and the result follows.

(2) Suppose $\bar{K}^{f}$ is a weakly $J$-ideal of $R\Join^{f}J$. Let
$x=f(a),y=f(b)\in S$ for $a,b\in R$ such that $0\neq xy\in K$ and $x\notin
J(S)$. Then $(0,0)\neq(a,f(a))(b,f(b))=(ab,f(ab))$ and similar to the proof of
(1), we have $(a,f(a))\notin J(R\Join^{f}J)$. By assumption, we have
$(b,f(b))\in\bar{K}^{f}$ and so $y=f(b)\in K$. Moreover, let $a,b\in R$ and
$j,k\in J$ such that $f(a)+j\notin J(S),f(b)+k\notin K$ and
$(f(a)+j)(f(b)+k)=0$. Suppose $ab\neq0$. Then $(0,0)\neq
(a,f(a)+j)(b,f(b)+k)=(ab,0)\in\bar{K}^{f}$. Similar to the proof of (1), we
conclude that $(a,f(a)+j)\notin J(R\Join^{f}J)$ and $(b,f(b)+k)\notin\bar
{K}^{f}$ which contradict the assumption that $\bar{K}^{f}$ is a weakly
$J$-ideal of $R\Join^{f}J$. Thus, $ab=0$ as needed.

Now, we assume $f(J(R))=J(S)+J$, $Ker(f)\subseteq J(R)$ and $K$ a weakly
$J$-ideal of $S$. Let $(a,f(a)+j),(b,f(b)+k)\in R\Join^{f}J$ such that
$(0,0)\neq(a,f(a)+j)(b,f(b)+k)=(ab,(f(a)+j)(f(b)+k))\in\bar{K}^{f}$ and
$(a,f(a)+j)\notin J(R\Join^{f}J)$. Then $f(a)+j\notin J(S)$ as in the proof of
(1) and $(f(a)+j)(f(b)+k)\in K$. We have two cases:

\textbf{Case I:} $(f(a)+j)(f(b)+k)\neq0$. In this case we conclude directly
that $(f(b)+k)\in K$. Thus, $(b,f(b)+k)\in\bar{K}^{f}$ and $\bar{K}^{f}$ is a
weakly $J$-ideal of $R\Join^{f}J$.

\textbf{Case II:} $(f(a)+j)(f(b)+k)=0$. If $f(a)+j\notin J(S)$ and
$f(b)+k\notin K$, then by assumption we should have $ab=0$. It follows that
$(a,f(a)+j)(b,f(b)+k)=(0,0)$, a contradiction. Therefore, again $\bar{K}^{f}$
is a weakly $J$-ideal of $R\Join^{f}J$.
\end{proof}

\begin{corollary}
\label{12}Let $R$ be a ring, $K$ a proper ideal of $R$ and $J$ a maximal ideal
of $R$. Then

\begin{enumerate}
\item If $\bar{K}=\left\{  (a,a+j):a\in R\text{, }j\in J\text{, }a+j\in
K\right\}  $ is a $J$-ideal of $R\Join J$, then $K$ is a $J$-ideal of $R$.
Moreover, the converse is true if $J\subseteq J(R)$.

\item If $\bar{K}$ is a weakly $J$-ideal of $R\Join J$, then $I$ is a weakly
$J$-ideal of $R$ and when $a+j\notin J(R),b+k\notin K$ with $a,b\in R$,
$j,k\in J$ and $ab+ak+bj+jk=0$, then $ab=0$. Moreover, the converse is true if
$J\subseteq J(R)$.
\end{enumerate}
\end{corollary}

In the following example, we prove that the condition $J\subseteq J(S)$ can
not be discarded in the proof of the converses of (1) and (2) in Theorem
\ref{9}.

\begin{example}
\label{14}Let $R=%
\mathbb{Z}
(+)%
\mathbb{Z}
_{4}$, $I=0(+)%
\mathbb{Z}
_{4}$ and $J=\left\langle 2\right\rangle (+)%
\mathbb{Z}
_{4}\nsubseteq J(R)$. Then $I$ is a weakly $J$-ideal of $R$ by Theorem \ref{9}
(2). Moreover, one can easily see that there are no $(r_{1},m_{1}%
),(r_{2},m_{2})\in R$ with $(r_{1},m_{1})(r_{2},m_{2})=(0,\bar{0})$, but
$(r_{1},m_{1})\notin J(R)$, $(r_{2},m_{2})\notin I$. Now, $((0,\bar
{1}),(2,\bar{1}))$,$((1,\bar{0}),(1,\bar{0})\in R\Join J$ with $((0,\bar
{1}),(2,\bar{1}))((1,\bar{0}),(1,\bar{0}))=((0,\bar{1}),(2,\bar{1}))\in I\Join
J\backslash((0,\bar{0}),(0,\bar{0}))$. Moreover, $((0,\bar{1}),(2,\bar
{1}))\notin J(R\Join J)$ since for example, $((0,\bar{1}),(2,\bar{1}%
)\notin\bar{Q}$ where $Q=\left\langle 3\right\rangle +%
\mathbb{Z}
_{4}\in Max(S)\backslash V(J)$. Since also clearly $((1,\bar{0}),(1,\bar
{0}))\notin I\Join J$, then $I\Join J$ is not a (weakly) $J$-ideal of $R\Join
J$.
\end{example}

Similarly, we justify in the following example that if $J\nsubseteq J(R)$,
then the converses of (1) and (2) of Corollary \ref{12} are not true in general.

\begin{example}
\label{15}$R=%
\mathbb{Z}
(+)%
\mathbb{Z}
_{4}$, $K=0(+)%
\mathbb{Z}
_{4}$ and $J=\left\langle 2\right\rangle +%
\mathbb{Z}
_{4}\nsubseteq J(R)$. Then $K$ is a (weakly) $J$-ideal of $R$. Moreover, if
for $a,b\in R$, $j,k\in J$, $(a+j,m_{1})\notin J(R)$ and $(b+k,m_{2})\notin
K$, then clearly $ab+ak+bj+jk\neq0$. Take $((2,\bar{0}),(0,\bar{1}%
))=((2,\bar{0}),(2,\bar{0})+(-2,\bar{1}))$,$((0,\bar{0}),(1,\bar{0}))\in
R\Join J$. Then $((2,\bar{0}),(0,\bar{1}))((0,\bar{0}),(1,\bar{0})=((0,\bar
{0}),(0,\bar{1}))\in\bar{K}\backslash((0,\bar{0}),(0,\bar{0}))$ since
$(0,\bar{1})\in K$. But, clearly, $((2,\bar{0}),(0,\bar{1}))\notin J(R\Join
J)$ and $((0,\bar{0}),(1,\bar{0})\notin\bar{K}$. Hence, $\bar{K}$ is not a
(weakly) J-ideal of $R\Join J$.
\end{example}

Even if $Ker(f)\subseteq J(R)$, the converse of (1) of Theorem \ref{10} need
not be true if $f(J(R))\neq J(S)+J$.

\begin{example}
\label{16}Let $R=%
\mathbb{Z}
(+)%
\mathbb{Z}
_{4}$, $S=%
\mathbb{Z}
$ and $J=\left\langle 2\right\rangle $ the ideal of $S$. Consider the
homomorphism $f:R\rightarrow S$ defined by $f((r,m))=r$. Note that
$J(S)=\left\langle 0\right\rangle $, $Ker(f)=0(+)%
\mathbb{Z}
_{4}=J(R)$ and $J(S)+J=J\neq f(J(R))$. Now, $K=\left\langle 0\right\rangle $
is a (weakly) $J$-ideal of $S$. Moreover, for $(r_{1},m_{1}),(r_{2},m_{2})\in
R$, $j,k\in J$ whenever $f(r_{1},m_{1})+j\notin J(S),f(r_{2},m_{2})+k\notin
K$, then $(f(r_{1},m_{1})+j)(f(r_{2},m_{2})+k)\neq0$. Take $((-2,0),0)$,
$((1,0),1)\in R\Join^{f}J$. Then $((-2,0),0)((1,0),1)=((-2,0),0)\in\bar{K}%
^{f}$ but $((-2,0),0)\notin J(R\Join^{f}J)$ and $((1,0),1)\notin\bar{K}^{f}$.
Therefore, $\bar{K}^{f}$ is not a (weakly) $J$-ideal of $R\Join^{f}J$.
\end{example}

We have proved in section 2 that if $I_{1}$ and $I_{2}$ are weakly $J$-ideals
of a ring $R$, then so is $I+J$. However, in the next example, we clarify that
the product $I_{1}I_{2}$ need not be a weakly $J$-ideal.

\begin{example}
\label{product}Let $R=%
\mathbb{Z}
(+)%
\mathbb{Z}
_{4}$ and $I=J=0(+)%
\mathbb{Z}
_{4}$. Now, $I$ is a weakly $J$-ideal of $R$ by Theorem \ref{9} (2) and
clearly there are no $(r_{1},m_{1}),(r_{2},m_{2})\in R$ with $(r_{1}%
,m_{1})(r_{2},m_{2})=(0,\bar{0})$, but $(r_{1},m_{1})\notin J(R)$,
$(r_{2},m_{2})\notin I$. Since also $J=J(R)$, then $I\Join J$ is a weakly
$J$-ideal of $R\Join J$ by Corollary \ref{11}. On the other hand, $\left(
I\Join J\right)  ^{2}=I^{2}\Join J=\left\langle (0,\bar{0})\right\rangle \Join
J$ is not a weakly $J$-ideal. Indeed, $((2,\bar{1}),(2,\bar{0}))$,$((0,\bar
{2}),(0,\bar{1})\in R\Join J$ with $((2,\bar{1}),(2,\bar{0}))((0,\bar
{2}),(0,\bar{1}))=((0,\bar{0}),(0,\bar{2}))\in I\Join J\backslash((0,\bar
{0}),(0,\bar{0}))$ but clearly $((2,\bar{1}),(2,\bar{0}))\notin J(R\Join J)$
and $((0,\bar{2}),(0,\bar{1})\notin I^{2}\Join J$.
\end{example}

Let $R$, $S$, $J$, $f$ and $I$ be as in Theorem \ref{9} and let $T$ be an
ideal of $f(R)+J$. As a general case of $I\Join^{f}J$, one can verify that if
$f(I)J\subseteq T\subseteq J$, then $I\Join^{f}T:=\left\{  (i,f(i)+t):i\in
I\text{, }t\in T\right\}  $ is an ideal of $R\Join^{f}J$. The proof of the
following result is similar to that of (1) of Theorem \ref{9} and left to the reader.

\begin{proposition}
\label{13}Let $R$, $S$, $J$, $f$, $I$ and $T$ as above. If $I\Join^{f}T$ is a
weakly $J$-ideal of $R\Join^{f}J$, then $I$ is a weakly $J$-ideal of $R$.
\end{proposition}

The following example shows that the converse of Proposition \ref{13} is not
true in general.

\begin{example}
\label{17}Let $R$, $S$, $J$ and $I$ be as in Example \ref{14} and let
$T=\left\langle 4\right\rangle (+)%
\mathbb{Z}
_{4}$. Then $IJ\subseteq T\subseteq J$ and $I\Join T$ is not a weakly J-ideal
of $R\Join^{f}J$. Indeed, $((0,\bar{1}),(4,\bar{1}))((1,\bar{0}),(1,\bar
{0})=((0,\bar{1}),(4,\bar{1}))\in I\Join T\backslash((0,\bar{0}),(0,\bar{0}))$
but clearly $((0,\bar{1}),(4,\bar{1}))\notin J(R\Join J)$ and $((1,\bar
{0}),(1,\bar{0}))\notin I\Join T$.
\end{example}

\end{document}